\documentclass{amsart}

\usepackage[dvipsnames]{xcolor}
\usepackage{hyperref}

\newtheorem{theorem}{Theorem}[section]
\newtheorem{lemma}[theorem]{Lemma}

\theoremstyle{definition}
\newtheorem{definition}[theorem]{Definition}

\newtheorem{corollary}[theorem]{Corollary}
\newtheorem{proposition}[theorem]{Proposition}

\theoremstyle{remark}
\newtheorem{remark}[theorem]{Remark}

\numberwithin{equation}{section}

\begin{document}

\title[Characterization of Relative Hyperbolicity]{A Characterization of Relative Hyperbolicity via Morse and Contracting Boundaries}

\author{Vyshnav PT}
\address{Department of Mathematical Sciences, Indian Institute of Science Education and Research Mohali, India}
\email{mp21013@iisermohali.ac.in}



\author{Pranab Sardar}
\address{Department of Mathematical Sciences, Indian Institute of Science Education and Research Mohali}
\email{psardar@iisermohali.ac.in}


\author{Rana Sardar}
\address{Department of Mathematical Sciences, Indian Institute of Science Education and Research Mohali}
\email{ranasardar@iisermohali.ac.in}

\thanks{The third author was supported by RA/Postdoctoral Fellowship \#IISER/25/Dean Faculty/2400} 


\subjclass[2020]{20F65, 20F67}

\date{}


\keywords{Morse boundary, Contracting boundary, compactness, Relatively hyperbolic groups}

\begin{abstract}
We prove the following boundary-theoretic characterization of relatively hyperbolic groups. Let $G$ be a finitely generated group with a finite collection $\mathcal{H}$ of finitely generated subgroups, and let $G^h$ denote the associated cusped space. We prove that the pair $(G,\mathcal{H})$ is non-elementary relatively hyperbolic if and only if the Morse boundary $\partial_M^{\mathcal{DL}} G^h$ or the contracting boundary $\partial_c^{\mathcal{FQ}} G^h$ is non-empty and compact. 

\end{abstract}

\maketitle


\section{Introduction}

The classes of hyperbolic and relatively hyperbolic groups play an important role in geometric group theory. The notion of relatively hyperbolic groups was introduced by Gromov and subsequently developed by Farb and Bowditch~\cite{Gro87, Farb, Bowditch}. 
Over the past decades, (relative) hyperbolicity has been studied from several different perspectives. Bowditch~\cite{Bowditch} characterizes relative hyperbolicity via group actions on fine hyperbolic graphs: if a group $G$ acts on a fine hyperbolic graph $K$ with finite edge stabilizers and finitely many edge orbits, and $\mathcal{P}$ is a collection of representatives of conjugacy classes of infinite vertex stabilizers, then $(G,\mathcal{P})$ is relatively hyperbolic. 
Yaman~\cite{Yaman2004} gave a characterization to relatively hyperbolic groups by showing that if a group acts as a convergence group on a perfect metrizable compactum \(M\), where every point is either a conical limit point or a bounded parabolic point, and the stabilizer of each bounded parabolic point is finitely generated, then the group is relatively hyperbolic. 
Gerasimov~\cite{Gerasimov2009} provided another characterization in terms of proper actions on pairs and cocompact actions on triples of a compactum. 
From a coarse geometric viewpoint, Druţu and Sapir~\cite{DrutuSapir2005}  obtain a characterization of relatively hyperbolic groups in terms of their asymptotic cones: such groups are relatively hyperbolic if and only if their asymptotic cones are tree-graded spaces with respect to \(\omega\)-limits of sequences of cosets of the subgroups H.
These results together provide a diverse collection of criteria for detecting relative hyperbolicity via boundary dynamics, asymptotic geometry, and group-theoretic properties.

In this paper, we provide a new characterization of relative hyperbolicity in terms of Morse and contracting boundaries.

Boundaries at infinity play a fundamental role in the study of finitely generated groups and their large-scale geometry. 
For word-hyperbolic groups, the Gromov boundary provides a powerful quasi-isometry invariant that encodes substantial geometric and dynamical information. 
In the broader setting of relatively hyperbolic groups, several boundary constructions and structural criteria have been developed. 
Beyond the hyperbolic setting, several quasi-isometric invariant boundary constructions capturing hyperbolic-like directions in general proper geodesic metric spaces have been introduced. 
A classical obstruction to naive boundary constructions is provided by the example of Croke and Kleiner~\cite{CrokeKleiner2000}, later strengthened by Wilson~\cite{Wilson2005}, showing that visual boundaries of CAT(0) spaces need not be quasi-isometry invariants. 
To address this issue, Charney and Sultan~\cite{CS2015} introduced the \emph{contracting boundary}, denoted by \(\partial_c X\), and Cordes~\cite{cordes2017} subsequently defined the \emph{Morse boundary}, denoted by \(\partial_M X\), for proper geodesic metric spaces \(X\). 
These boundaries isolate geodesic rays exhibiting hyperbolic-like behaviour and are invariant under quasi-isometries. 
As sets, the contracting and Morse boundaries coincide~\cite[Theorem~2.2]{CM2019}. 

Different natural quasi-isometric invariant topologies on these boundaries have been studied. 
In particular, the direct limit topology, denoted by \(\mathcal{DL}\), for Morse boundary \cite{CS2015,cordes2017}.  On the other hand, the fellow-travelling quasigeodesic topology, denoted by \(\mathcal{FQ}\), for contracting boundary introduced by Cashen and Mackay~\cite{CM2019}. This topologies provide useful frameworks for analysing their structure. 
An important consequence of this theory is a purely boundary-theoretic characterization of word-hyperbolic groups: compactness of the Morse boundary \(\partial_M^\mathcal{DL} G\) or the contracting boundary \(\partial_c^\mathcal{FQ} G\) characterizes hyperbolicity~\cite{CordesDurham2019,CM2019} of a finitely generated group \(G\).

The aim of this article is to extend this boundary-theoretic characterization to the setting of relatively hyperbolic groups. More precisely, we show that relative hyperbolicity can also be detected through compactness of either a Morse or a contracting boundary. In this sense, our result contributes a new intrinsic boundary characterization to the existing list of criteria for relative hyperbolicity. 

Let $G$ be a finitely generated group together with a finite collection $\mathcal{H}$ of finitely generated subgroups, and let $G^h$ denote the associated cusped space. 
In~\cite{palpandey2024}, Pal and Pandey proved that if the contracting boundary $\partial_c^{\mathcal{FQ}} G^h$ is compact and the combinatorial horoballs in $G^h$ are contracting, then $(G,\mathcal{H})$ is relatively hyperbolic. 

In the present work, we remove the additional hypothesis that the combinatorial horoballs are contracting, and prove that compactness of the Morse or contracting boundary of the cusped space alone characterizes relative hyperbolicity. This provides a sharp and intrinsic boundary-theoretic criterion for relative hyperbolicity. Our main result is the following.

\begin{theorem}\label{thm_1}
    Suppose  $G$ is a finitely generated group that is not virtually cyclic and  $\mathcal{H}$ is a collection of finitely generated, infinite index, infinite subgroups of $G$. Let $G^h$ denote the cusped space associated to the pair $(G,\mathcal{H})$. Then the following are equivalent:
    \begin{enumerate}
        \item The pair $(G,\mathcal{H})$ is relatively hyperbolic;
        \item The Morse boundary $\partial_M^{\mathcal{DL}} G^h$ is non-empty compact;
        \item The contracting boundary $\partial_c^{\mathcal{FQ}} G^h$ is non-empty compact.
    \end{enumerate}
\end{theorem} 

Note that the non-emptiness assumption in Theorem~\ref{thm_1} is necessary. For example, if $G=\mathbb{Z}^2$ and $H=\mathbb{Z}$, then $G$ is not hyperbolic relative to $H$, while $\partial_M G^h$ is empty and hence compact. The contracting boundary of $G^h$ consists of exactly one point if and only if $\mathcal{H}=\{G\}$. In this case $G$ is relatively hyperbolic with respect to itself.
The contracting boundary $\partial_c G^h$ is compact and consists of exactly two points if and only if the group \(G\) is virtually cyclic and the collection  $\mathcal{H}$ is empty. Consequently, in proving Theorem~\ref{thm_1} we may assume that every subgroup in $\mathcal{H}$ is of infinite index in $G$ and that the Morse (equivalently, contracting) boundary contains at least three points.




\section{Morse Boundary and Contracting Boundary}

In this section, we recall the notions of the Morse boundary and the contracting boundary of a geodesic metric space. These boundaries capture hyperbolic–like directions in spaces that may not be globally hyperbolic. 

\subsection{Morse Boundary}

\begin{definition}[Morse (quasi) geodesic]
    A (quasi) geodesic $\gamma$ in a geodesic metric space $X$ is called \emph{$N$-Morse} if there exists a function (not necessarily continuous or non-decreasing) \(N: \mathbb{R}_{\ge 1} \times \mathbb{R}_{\ge 0} \to \mathbb{R}_{\ge 0}\) such that whenever $\alpha$ is a $(K,C)$–quasi-geodesic with endpoints on $\gamma$, the image of $\alpha$ is contained in the $N(K,C)$–neighborhood of $\gamma$. The function $N$ is called the \emph{Morse gauge} of $\gamma$.
\end{definition}

Intuitively, Morse geodesics are geodesics that remain stable under quasi-geodesic perturbations. In hyperbolic spaces, every geodesic is Morse with a uniform gauge. In contrast, Euclidean spaces contain no Morse geodesic rays. Morse quasi-geodesics appear naturally in many geometric contexts, including quasi-geodesics in hyperbolic spaces and axes of pseudo-Anosov mapping classes in Teichm\"uller space~\cite{Minsky1996}.

A key feature of Morse geodesics is their stability with respect to quasi-geodesics with the same endpoints.

\begin{lemma}[Morse Quasi-Geodesic Stability, {\cite[Lemma~2.1]{cordes2017}}] \label{lem_Morse_quasi-geodesic_stability}
    Let $X$ be a geodesic metric space and let $\alpha \colon [0,l_1] \to X$ be an $N$–Morse geodesic segment. Let $\beta \colon [0,l_2] \to X$ be a $(\lambda,\varepsilon)$–quasi-geodesic such that $\alpha(0)=\beta(0)$ and $\alpha(l_1)=\beta(l_2)$. Then the Hausdorff distance between $\alpha$ and $\beta$ is bounded by \(2N(\lambda,\varepsilon') + (\lambda+\varepsilon)\), where $\varepsilon' = 2(\lambda+\varepsilon)$. If $\beta$ is continuous, then the Hausdorff distance is bounded by $2N(\lambda,\varepsilon)$.
\end{lemma}

Another important property is that triangles formed by Morse geodesics are slim.

\begin{lemma}[{\cite[Lemma~2.2]{cordes2017}}] \label{lem_4N_slim_triangle}
    Let $X$ be a geodesic metric space and let $\alpha_1 \colon [0,l_1] \to X$ and $\alpha_2 \colon [0,l_2] \to X$ be $N$–Morse geodesics such that $\alpha_1(0)=\alpha_2(0)=p$. Let $\gamma$ be a geodesic joining $\alpha_1(l_1)$ and $\alpha_2(l_2)$. Then the geodesic triangle $\alpha_1 \cup \alpha_2 \cup \gamma$ is $4N(3,0)$–slim.
\end{lemma} 

We now describe the construction of the Morse boundary. Fix a basepoint $p\in X$. For a given Morse gauge $N$, consider the subset of the Morse boundary consisting of rays with Morse gauge bounded by $N$:
\[
\partial_M^{N} X_p
=
\left\{
[\alpha] \mid
\exists \beta \in [\alpha] \text{ such that } 
\beta \text{ is an } N\text{–Morse geodesic ray with } \beta(0)=p
\right\}.
\]
The topology on $\partial_M^N X_p$ is defined using a neighborhood basis given by fellow traveling of geodesic rays.

\begin{lemma}[{\cite[Lemma~3.3]{cordes2017}, \cite[Definition~2.9]{Cordes2024}}]
    Let $\alpha$ be an $N$–Morse geodesic ray with $\alpha(0)=p$. For each positive integer $n$, let $V_n(\alpha)$ be the set of geodesic rays $\gamma$ such that
    \[
    \gamma(0)=p \quad \text{and} \quad d(\alpha(t),\gamma(t)) < \delta_N \quad \text{for all } t<n,
    \]
    where $\delta_N > 4N(3,0)$. Then $\{V_n(\alpha)\}_{n\in\mathbb N}$ forms a fundamental system of neighborhoods of $[\alpha]$ in $\partial_M^N X_p$.
\end{lemma}

The spaces $\partial_M^N X_p$ naturally fit into a directed system as the Morse gauge increases.

\begin{lemma}[{\cite[Lemma~2.7]{Cordes2024}}]
    Let $N$ and $N'$ be Morse gauges such that every $N$-Morse geodesic is also $N'$-Morse. Then the natural inclusion \(i : \partial_M^{N} X_p \to \partial_M^{N'} X_p\) is continuous. In particular, this holds whenever $N \le N'$.
\end{lemma}

Using these inclusions, we define the Morse boundary as a direct limit.

\begin{definition}[{\cite[Definition~2.9]{Cordes2024}}]
    Let $X$ be a proper geodesic metric space and let $p \in X$. The \emph{Morse boundary} of $X$ is defined by
    \[
    \partial_M X_p = \varinjlim_{N \in \mathcal{M}} \partial_M^N X_p,
    \]
    equipped with the induced direct limit topology $\mathcal{DL}$.
\end{definition}

The Morse boundary enjoys several important geometric properties.

\begin{lemma}[{\cite[Theorem~2.16]{cordes2019}, \cite{cordes2017}}]
Let $X$ be a proper geodesic space. Then the Morse boundary $\partial_M X_p$:

\begin{enumerate}
\item is independent of the choice of basepoint;

\item is a visibility space, i.e.\ any two points can be joined by a bi-infinite Morse geodesic;

\item is invariant under quasi-isometries;

\end{enumerate}
\end{lemma}

\medskip
\subsection{Contracting Boundary}

We now turn to contracting geodesics, which provide another way to describe hyperbolic-type behavior in metric spaces.

\begin{definition}[Contracting quasi-geodesic]
    Let $X$ be a geodesic metric space and let $\gamma:I\to X$ be a quasi-geodesic. Define the closest point projection \(\pi_\gamma:X\to \mathcal{P}(\gamma(I))\) by
    \[
    \pi_\gamma(x)=\{z\in\gamma(I)\mid d(x,z)=d(x,\gamma(I))\}.
    \]

    A function $\rho:[0,\infty)\to[0,\infty)$ is called \emph{sublinear} if it is non-decreasing, eventually non-negative, and
    \[
    \lim_{r\to\infty}\frac{\rho(r)}{r}=0.
    \]

    The quasi-geodesic $\gamma$ is called \emph{$\rho$-contracting} if
    \[
    d(x,y)\le d(x,\gamma(I)) \quad\Rightarrow\quad \operatorname{diam}(\pi_\gamma(x)\cup\pi_\gamma(y)) \le \rho(d(x,\gamma(I))). 
    \]

    If such a sublinear function $\rho$ exists, we say that $\gamma$ is \emph{contracting}.
\end{definition}

A fundamental result shows that Morse and contracting subsets are equivalent notions.

\begin{theorem}[{\cite[Theorem~1.4]{ACGH2017}, \cite[Theorem~2.2]{CM2019}}]
    Let $Z$ be a subset of a geodesic metric space $X$. Then the following are equivalent:

    \begin{enumerate}
        \item $Z$ is Morse;
        \item $Z$ is contracting.
    \end{enumerate}

    Moreover, the defining functions of the two properties determine each other effectively.
\end{theorem}

For each $\zeta\in\partial_c X$ one can associate a canonical contraction function. Define
\[
\rho_\zeta(r)
=
\sup_{\alpha,x,y}
\operatorname{diam}\big(\pi_\alpha(x)\cup\pi_\alpha(y)\big),
\]
where the supremum is taken over geodesic representatives $\alpha\in\zeta$ and points $x,y\in X$ satisfying\(d(x,y)\le d(x,\alpha)\le r\). The function $\rho_\zeta$ is sublinear and every geodesic representing $\zeta$ is $\rho_\zeta$-contracting; see \cite[Lemma~5.2]{CM2019}. Using contracting geodesics, one defines the contracting boundary $\partial_c X$, consisting of contracting quasi-geodesic rays modulo Hausdorff equivalence. 
Cashen and Mackay~\cite{CM2019} introduced a topology on this boundary based on the notion of fellow travelling of quasi-geodesics. 

\begin{definition}[Topology of fellow traveling quasi-geodesics, {\cite{CM2019}}]
    Let $X$ be a proper geodesic metric space and fix a basepoint $o\in X$. For $r\ge1$, we define \(N_r^c(o)=\{x\in X\mid d(o,x)\ge r\}.\) 
    Let $\zeta\in\partial_c X$ and fix a geodesic ray $\alpha_\zeta\in\zeta$. For $r\ge1$ define $U(\zeta,r)$ to be the set of points $\eta\in\partial_c X$ such that for every $L\ge1$, $A\ge0$, and every continuous $(L,A)$–quasi-geodesic ray $\beta\in\eta$ we have
\[
d\big(\beta,\alpha_\zeta\cap N_r^c(o)\big)\le k(\rho_\zeta,L,A),
\]
where
\(
k(\rho,L,A)
=
\max\left\{
3A,\,
3L^2,\,
1+\inf\{R>0\mid 3L^2\rho(r)\le r \text{ for } r\ge R\}
\right\}.
\)

A subset $U\subset\partial_c X$ is declared to be open if for every $\zeta\in U$ there exists $r\ge1$ such that
\[
U(\zeta,r)\subset U.
\]
The contracting boundary equipped with this topology is called the \emph{topology of fellow–travelling quasi-geodesics} and is denoted by $\partial_c^{\mathcal{FQ}}X$.
\end{definition} 

Cashen and Mackay~\cite[Proposition~5.5, and~5.11]{CM2019} showed that the sets $\{U(\zeta,r)\mid r\ge1\}$ form a neighbourhood basis at $\zeta$, and that the resulting topology is independent, up to homeomorphism, of the choice of basepoint and of the choice of representative geodesic rays. Moreover, this topology is also quasi-isometric invariant.

In general, the topology on the contracting boundary is weaker than the direct limit topology on the Morse boundary.

\begin{lemma}
    The identity map \(id : \partial_M^\mathcal{DL} X \longrightarrow \partial_c^\mathcal{FQ} X\) is continuous and bijective. 
\end{lemma}

Consequently, we obtain the following.

\begin{corollary}
    If $\partial_M^\mathcal{DL} X$ is compact, then $\partial_c^\mathcal{FQ} X$ is also compact.
\end{corollary}

An application of Morse or contracting boundaries is the following boundary characterization of hyperbolic groups, due to Cordes--Durham and Cashen--Mackay.

\begin{theorem}[\cite{CordesDurham2019,CM2019}] \label{thm_characterization_hyp_gp}
    Let $G$ be a finitely generated infinite group. Then the following are equivalent:

    \begin{enumerate}
        \item $G$ is hyperbolic;
        \item the Morse boundary $\partial_M^\mathcal{DL} G$ is non-empty and compact;
        \item the contracting boundary $\partial_c^\mathcal{FQ} G$ is non-empty and compact.
    \end{enumerate}
\end{theorem}

\medskip
\subsection{Stability of Weak Hull}

We now discuss stability properties of subsets determined by the Morse boundary.

\begin{definition}[Stable subspaces {\cite{CordesHume2017}}]
    Let $X$ be a geodesic metric space. A quasi-convex subspace $Y\subset X$ is called \emph{$N$-stable} if every pair of points in $Y$ can be joined by an $N$-Morse geodesic in $X$.
\end{definition}

Cordes and Hume \cite{CordesHume2017} introduced subsets $X_e^{(N)}$ consisting of points that can be connected to a basepoint by $N$-Morse geodesics; that is,
\[
    X_e^{(N)} = \{\, y \in X \mid \text{there exists an $N$--Morse geodesic segment } [e,y] \subset X \,\}.
\]
This family of subsets $X_e^{(N)}$ satisfies several important properties.
\begin{theorem}[{\cite[Theorem~A]{CordesHume2017}}] \label{thm_properties_of_stable_sets}
    Each $X_e^{(N)}$ is hyperbolic and stable in \(X\). Moreover, every stable subset of \(X\) is a quasi–convex subset of some  $X_e^{(N)}$. 
\end{theorem}

Using Morse (or contracting) boundaries, one can define the weak hull of a subset of the boundary.

\begin{definition}[Weak hull]
    Let $X$ be a proper geodesic metric space and let $K\subset\partial_M X$ contain at least two points. The \emph{weak hull} $\mathfrak h(K)$ is the union of all bi-infinite geodesics whose endpoints lie in $K$.
\end{definition}

The weak hull of a compact subset of the Morse boundary turns out to be hyperbolic.

\begin{proposition}\label{prop_weakhull_hyperbolic}
    Let $X$ be a proper geodesic metric space and let $K \subset \partial_M X$ be a compact subset containing at least two points. Then: 
    \begin{enumerate}
        \item For every $p \in X$, there exists a Morse gauge $N$ such that $K \subset \partial_M X_p^{(N)}$.
        \item The weak hull $\mathfrak{h}(K)$ is a stable subset of $X$.
        \item There exists a Morse gauge $N'$ such that $\mathfrak{h}(K)$ is a quasiconvex subset of $X_p^{(N')}$, and hence $\mathfrak{h}(K)$ is Gromov hyperbolic.
    \end{enumerate}
\end{proposition}

\begin{proof}
    Statement (i) and (ii) follows from Lemma~4.1  and Lemma~4.2 of \cite{CordesDurham2019} respectively.  
    In Lemma~4.1  and Lemma~4.2 of \cite{CordesDurham2019} the authors have assumed that the space admits an isometric group action, but their proof also holds without this assumption.
    Statement (iii) follows from Theorem~\ref{thm_properties_of_stable_sets}.
\end{proof}

However, this phenomenon fails for the contracting boundary.


\begin{remark}
    Proposition~\ref{prop_weakhull_hyperbolic} does not hold in general for contracting boundaries. Consider the group
    \[
    G=\mathbb{Z} * \mathbb{Z}^2, \quad \text{where} \quad \mathbb{Z}= \langle a \rangle, \quad \text{and} \quad \mathbb{Z} \times \mathbb{Z} = \langle b,c \mid bcb^{-1}c^{-1}=1 \rangle .
    \]
    For each $k\in\mathbb N$ let
    \[
    \gamma_k = a^k b^k a a a \cdots .
    \]
    Then each $\gamma_k$ is contracting, and the sequence $\gamma_k(\infty)$ converges to $\gamma(\infty)$ in $\partial_c^{\mathcal FQ}G$, where $\gamma=a a a\cdots$. The set \(\{\gamma(\infty)\} \cup \{\gamma_k(\infty) \mid k\in\mathbb{N}\}\) is a is compact subset of $\partial_c^{\mathcal FQ}G$, and the corresponding weak hull is also hyperbolic, but is not stable in \(G\).

    On the other hand, consider the sequence $\{\beta_k\}$ defined by 
    \[
    \beta_k = a^kb^kc^kaaa\cdots
    \]
    Each $\beta_k$ is contracting, and the sequence $\{\beta_k(\infty)\}$ also converges to $\gamma(\infty)$ in $\partial_c^{\mathcal FQ}G$. Although the set \(\{\beta(\infty)\} \cup \{\gamma_k(\infty) \mid k\in\mathbb{N}\} \subset \partial_c^{\mathcal FQ}G\) is compact in $\partial_c^{\mathcal FQ}G$, the corresponding weak hull is neither hyperbolic nor stable in \(G\). 
\end{remark}

\section{Pal-Pandey Characterizations of Relative Hyperbolicity}

In this section, after recalling the Groves--Manning definition of relatively hyperbolic groups via combinatorial horoballs, we review a result of Pal-Pandey \cite{palpandey2024} involving the contracting boundary, and then establish an alternative characterization in terms of the Morse boundary.

\subsection{Relatively hyperbolic groups}

The notion of relatively hyperbolic groups was introduced by Gromov and initially developed by Farb and Bowditch \cite{Gro87, Farb, Bowditch}. We refer the reader to the survey article of Hruska~\cite{Hruska} for a discussion of the various equivalent definitions of relative hyperbolicity. In this article we adopt the definition due to Groves and Manning~\cite{GM2008}, which is formulated using combinatorial horoballs and the associated cusped space.

\begin{definition}[Combinatorial horoball, \cite{GM2008}] \label{defn_horoball}
    Let $H$ be a locally finite graph with vertex set $V(H)$ and all edges of length one. Let $d_H$ denote the associated path metric on $H$. The combinatorial horoball based on $H$, denoted by $\mathcal{H}(H,d_H)$, is the graph defined as follows:
    (1)  The vertex set is $\mathcal{H}^{(0)} :=  V(H) \times \{\mathbb{N}\cup\{0\}\}$. \\ 
    (2) The edge set $\mathcal{H}^{(1)}$ consists of two types of edges:
    \begin{itemize}
        \item Vertical edges: For each $n \in \mathbb{N}\cup \{0\}$ and $x \in V(H)$, there is an edge between $(x, n)$ and $(x, n+1)$. 

        \item Horizontal edges: For each $x, y \in V(H)$,
        \begin{itemize}
            \item if \(e\) is an edge of $H$ joining \(x\) to \(y\), then there is a corresponding edge \(\overline{e}\) connecting \((x,0)\) to \((y,0)\), 

            \item  for each $n \in \mathbb{N}$, if $0 < d_H(x, y) \leq 2^n$, then there is a single edge between $(x, n)$ and $(y, n)$. 
        \end{itemize}
    \end{itemize}    
\end{definition} 

All edges of $\mathcal{H}(H,d_H)$ are assigned unit length, and the graph is equipped with the induced path metric. With this metric, $\mathcal{H}(H,d_H)$ is a hyperbolic metric space (see~\cite{GM2008}). For convenience, we denote this hyperbolic metric space by $(H^h,d_{H^h})$ and refer to it simply as a \emph{horoball}. 

\begin{definition}[Relatively hyperbolic group] \label{defn_rel_hyp}
    Let $G$ be a finitely generated group and let \(\mathcal{H}=\{H_1,H_2,\dots,H_n\}\) be a finite collection of finitely generated infinite subgroups of $G$. Choose a finite generating set $S$ of $G$ such that, for each $i=1,\dots,n$, the intersection $H_i\cap S$ generates $H_i$. Let $X$ denote the Cayley graph of $G$ with respect to the generating set $S$.

    The space obtained by attaching a combinatorial horoball $H^h$ to each left coset of every subgroup $H_i\in\mathcal{H}$ is denoted by $X^h$ and is called the \emph{cusped space} associated to the pair $(G,\mathcal{H})$. The left cosets of the subgroups in $\mathcal{H}$ inside $X$ are referred to as \emph{horospheres} (or horosphere-like sets).

    We say that $G$ is \emph{$\delta$-hyperbolic relative to $\mathcal{H}$}, or simply that the pair $(G, \mathcal{H})$ is a \emph{$\delta$-relatively hyperbolic group} if the cusped space $X^h$ is a $\delta$-hyperbolic metric space. The group $G$ is said to be \emph{hyperbolic relative to $\mathcal{H}$} if there exists $\delta \ge 0$ such that $(G,\mathcal{H})$ is $\delta$-relatively hyperbolic.
\end{definition} 

Note that all groups appearing in the pair $(G,\mathcal{H}_G)$ are finitely generated, and hence the associated cusped space is locally finite and proper (see~\cite[Remark~3.14]{GM2008}). 

\begin{lemma}[{\cite[Lemma~3.5]{palpandey2024}}]
    Let \(G\) be a finitely generated group and let \(\mathcal{H}\) be a finite collection of finitely generated, infinite subgroups of infinite index in \(G\). Let \(G^h\) denote the cusped space associated to the pair \((G,\mathcal{H})\). \(\partial_c G^h\) is metrizable. 
\end{lemma} 

The following result, due to Pal--Pandey, establishes that the compactness of the contracting boundary together with the condition that every vertical ray is contracting provides a characterization of relative hyperbolicity.

\begin{theorem}[{\cite[Theorem~1.1]{palpandey2024}}] \label{thm_pal_pandey}
    Let $G$ be a finitely generated group and let 
    \(
    \mathcal{H}=\{H_i\}
    \)
    be a finite collection of finitely generated infinite index subgroups of $G$. Then $G$ is hyperbolic relative to $\mathcal{H}$ if and only if every combinatorial horoball $\mathcal{H}^h_i$ is contracting in the cusped space $G^h$, and the cusped space $G^h$ has compact contracting boundary.
\end{theorem} 

The proof of Theorem~\ref{thm_pal_pandey} is somewhat technical. However, it becomes simpler if one considers the Morse boundary of the cusped space \(G^h\) to be compact, rather than working with the contracting boundary.

\begin{theorem}\label{prop_condition_to_be_rel_hyp}
    Let $G$ be a finitely generated group and let $\mathcal{H}=\{H_1,\dots,H_n\}$ be a finite collection of finitely generated infinite subgroups of $G$. The pair $(G,\mathcal{H})$ is relatively hyperbolic if and only if the Morse boundary $\partial_M G^h$ is compact and every vertical rays in $G^h$ are Morse.
\end{theorem}

\begin{proof}
    If the pair $(G,\mathcal H)$ is relatively hyperbolic, then the Morse boundary $\partial_M G^h$ coincides with the Bowditch boundary; in particular, it is compact and contains all limit points arising from the cusps of $G^h$. Conversely, suppose that $\partial_M G^h$ is compact and contains all limit points arising from the cusps of $G^h$. By Proposition~\ref{prop_weakhull_hyperbolic}, the weak hull \(\mathfrak{h}(\partial_M G^h) \subset G^h\) is Gromov hyperbolic. We claim that $\mathfrak{h}(\partial_M G^h) = G^h$, which will complete the proof.
    
    Because $\partial_M G^h$ contains at least three points, consider a geodesic $[a,b]$ joining two distinct points $a,b \in \partial_M G^h$. This geodesic must pass through $G$, and let $g \in G \cap [a,b]$. Then $g \in \mathfrak{h}(\partial_M G^h)$, and by $G$ acting isometrically on itself, we deduce that \(G \subset \mathfrak{h}(\partial_M G^h).\)

    Now, let $i = 1, \dots, n$, and consider any $g \in G$. Let $a_i \in \mathfrak{h}(\partial_M G^h)$ be the limit point of the left coset $gH_i$. For any $b \in \mathfrak{h}(\partial_M G^h)$ distinct from $a_i$, consider a geodesic $[a_i,b]$ and let $h_i \in gH_i \cap [a_i,b]$. Then $[h_i,a_i] \subset \mathfrak{h}(\partial_M G^h)$.
    Since the subgroup $gH_ig^{-1}$ acts by isometries on $G$ and preserves $gH_i$, all vertical rays in the horoball $(gH_i)^h$ lie inside $\mathfrak{h}(\partial_M G^h)$. Therefore, \((gH_i)^h \subset \mathfrak{h}(\partial_M G^h)\). Since $i$ and $g$ were chosen arbitrary, for every $i = 1, \dots, n$ and every $g \in G$, the horoball $(gH_i)^h$ is contained in $\mathfrak{h}(\partial_M G^h)$. Hence, we conclude that \(G^h = \mathfrak{h}(\partial_M G^h),\) as required.
\end{proof}

\section{Proof of Theorem~\ref{thm_1}}

Throughout this section, we assume that $G$ is a finitely generated group and that \(\mathcal{H}=\{H_1,H_2,\ldots,H_n\}\) is a finite (possibly empty) collection of finitely generated subgroups of infinite index in $G$. Let $G^h$ denote the cusped space associated to the pair $(G,\mathcal{H})$, and let $\partial_M G^h$ and \(\partial_c G^h\) denote its Morse and contracting boundary, respectively.

Suppose that the contracting boundary $\partial_c G^h$ is compact and contains two points. Then $G$ acts on $\mathfrak{h}(\partial_c G^h)$ quasiisometric to $\mathbb{R}$ and hence is virtually cyclic.If $\mathcal{H}$ is non-empty, then the only possible geodesic rays in $G^h$ are vertical rays. This is not possible since if $H$ is a proper infinite subgroup of $G$ then $H$ and $gH$ are Hausdorff close which contradicts that the vertical ray of $H$ is contracting.

\begin{theorem} \label{thm_all_vertical_rays_are_contracting}
    Suppose that the contracting boundary $\partial_c G^h$ is compact and contains at least three points. Then all vertical geodesic rays in $G^h$ are uniformly contracting.
\end{theorem}

\begin{proof} 
    We argue by contradiction. Up to renaming the subgroups if necessary, suppose that the contracting boundary $\partial_c G^h$ contains all vertical rays arising from the cusps of the cosets of $H_1,\dots,H_i$, where $0\le i<n$, and does not contain the vertical rays arising from the cusps of the cosets of $H_{i+1},\dots,H_n$. Let $G_1^h$ denote the cusped space corresponding to the pair $(G,H_1,\dots,H_i)$.

    \medskip

    \noindent
    \textbf{Claim.} There exists $R\ge 0$ such that the weak hull \(\mathfrak{h}(\partial_c G^h)\) is contained in the \(R\)-neighbourhood of \(G_1^h\) with respect to the metric $d_{G^h}$, that is, 
    \[
    \mathfrak{h}(\partial_c G^h)\subset N_R(G_1^h).
    \]

    \medskip

    \noindent
    Suppose, for a contradiction, that this is false. Then for each $k\ge 1$ there exists a point \(z_k\in \mathfrak{h}(\partial_c G^h)\) such that
    \[
    z_k\in (g_kH_{i_k})^h 
    \text{ for some }g_k\in G,\, i_k\in\{i+1,\dots,n\},\quad \text{and} \quad d_{G^h}(z_k,G_1^h)>k.
    \]

    Passing to a subsequence, we may assume $i_k=i+1$ for all $k$. Then
    \[
    g_k^{-1}z_k\in H_{i+1}^h \quad \text{and} \quad
    d_{G^h}(g_k^{-1}z_k,H_{i+1}) \geq d_{G^h}(g_k^{-1}z_k,G_1^h)>k.
    \]

    Since the action of $G$ on $G^h$ preserves the weak hull $\mathfrak{h}(\partial_c G^h)$, we have
    \[
    g_k^{-1}z_k\in \mathfrak{h}(\partial_c G^h).
    \]
    Hence there exists a bi-infinite geodesic $\gamma_k'$ in $G^h$ passing through $g_k^{-1}z_k$ whose endpoints lie in $\partial_c G^h$. Let $h_k\in H_{i+1}$ be a point where $\gamma_k'$ intersects $H_{i+1}$. After reparametrization, we may assume
    \[
    g_k^{-1}z_k\in [h_k,\gamma_k'(\infty)).
    \]
    Left translating by $h_k^{-1}$ we obtain
    \[
    h_k^{-1}g_k^{-1}z_k\in [1,h_k^{-1}\gamma_k'(\infty)).
    \]

    Thus there exist sequences $\{w_k\}\subset H_{i+1}$ and $\{a_k\}\subset\partial_c G^h$ and geodesic rays \(\gamma_k=[1,a_k)\) such that
    \[
    w_k\in \gamma_k,\quad [1,w_k] \subset H_{i+1}^h \quad \text{and} \quad d_{G^h}(w_k,G_1^h)>k .
    \] 
    Since $\partial_c G^h$ is compact, after passing to a subsequence we may assume that \(a_k\to a\in \partial_c G^h\). 

    Consider the vertical geodesic ray \(\gamma\) in $H_{i+1}^h$ based at $1$. Then the rays $\gamma_k$ converge uniformly on compact sets to $\gamma$. Consequently,
    \[
    \gamma(\infty) = a \in \partial_c G^h .
    \] 
    However, no vertical ray in $H_{i+1}^h$ is contracting. This contradiction proves the claim.

    \medskip

    Therefore, $\mathfrak{h}(\partial_c G^h)$ lies in a bounded neighbourhood of $G_1^h$. It follows that the Hausdorff distance between $G_1^h$ and $\mathfrak{h}(\partial_c G^h)$ is finite in \(G^h\). 
    As \(G_1^h\) is properly embedded in \(G^h\), the metrics $d_{G_1^h}$ and $d_{G^h}$ on $G_1^h$ are quasi-isometric. Hence, the contracting boundary $\partial_c(G_1^h,d_{G_1^h})$ is homeomorphic to $\partial_c G_1^h$, and therefore is compact.

    \medskip 

    \noindent
    \textbf{Case 1.} Suppose $i=0$, so that $G_1^h=G$.

    Then $\partial_c G$ is compact. By Theorem~\ref{thm_characterization_hyp_gp} it follows that $G$ is hyperbolic. For any fix \(j \in \{1,\dots,n\}\), since $H_j$ is an infinite subgroup of $G$, it contains an element $h$ of infinite order. Consider the path
    \[
    \alpha=[1,h]\cup[h,h^2]\cup[h^2,h^3]\cup\cdots
    \]
    in $G$. Since $G$ is hyperbolic, the orbit $\{h^n\}$ forms a quasi-geodesic in $G$, and hence $\alpha$ is a quasi-geodesic in $G$. However, in the cusped space $G^h$, the path $\alpha$ stays within a bounded neighbourhood of the horoball corresponding to the coset of $H_j$ and travels essentially horizontally. In particular, $\alpha$ is also a quasi-geodesic in $G^h$. Also, the distance between $\alpha(0)$ and $\alpha(n)$ in the metric $d_{G^h}$ grows only logarithmically in $n$, while the parameter length of $\alpha$ grows linearly. Thus $\alpha$ cannot be a quasi-geodesic in $G^h$, a contradiction.

    \medskip 

    \noindent
    \textbf{Case 2.} Suppose $i\neq 0$.

    Since $\partial_c G_1^h$ is compact and contains at least three points and every vertical rays in $G_1^h$ are contracting, it follows from Theorem~\ref{thm_pal_pandey} that $G$ is hyperbolic relative to $H_1,\dots,H_i$, that is, $(G_1^h,d_{G_1^h})$ is hyperbolic.

    Now for each $j\in\{i+1,\dots,n\}$ the subgroup $H_j$ acts by isometries on the hyperbolic space $(G_1^h,d_{G_1^h})$. Since $H_j$ is infinite, its orbit in $G_1^h$ is unbounded. Hence, the limit set $\Lambda(H_j)$ in $\partial G_1^h$ is nonempty. Suppose that $|\Lambda(H_j)|=1$. Then $H_j$ acts parabolically on $G_1^h$, fixing a unique point of $\partial G_1^h$. However, $H_j$ is not one of the peripheral subgroups used in the construction of $G_1^h$, and therefore it cannot act parabolically. Thus $|\Lambda(H_j)|\neq 1$. Consequently, \(|\Lambda(H_j)|\ge 2\). It follows that $H_j$ contains a loxodromic element (see \cite[Theorem~2.5]{Osin2018}). Let \(h \in H_j\) be a loxodromic element. Consider the path
    \[
    \alpha=[1,h]\cup[h,h^2]\cup[h^2,h^3]\cup\cdots
    \]
    in $G_1^h$. Since $h$ acts loxodromically, the orbit $\{h^n\}$ is a quasi-geodesic in $G_1^h$, and therefore $\alpha$ is a quasi-geodesic ray in $G_1^h$. As the metrics $dG_1^h$ and $dG^h$ are quasiisometric, we get that $\alpha$ is a quasigeodesic ray in $G_1^h$ with respect to the metric $d_{G^h}$ and hence \(\alpha\) is a quasigeodesic in $G^h$. 

    However, when viewed in the cusped space $G^h$, the path $\alpha$ lies inside the horoball corresponding to the coset of $H_j$ and travels essentially horizontally. Consequently, the distance between $\alpha(0)$ and $\alpha(n)$ in the metric $d_{G^h}$ grows only logarithmically in $n$, while the parameter length of $\alpha$ grows linearly. Thus, $\alpha$ fails to be a quasi-geodesic in $G^h$, again a contradiction.

    \medskip

    Both cases lead to contradictions. Therefore, every vertical geodesic ray in $G^h$ must be contracting. Every vertical geodesic ray based at cosets of the same subgroups is uniformly Morse. Since the collection $\mathcal{H}$ contains only finitely many subgroups, the contraction functions can be chosen uniformly. This completes the proof.     
\end{proof}

\begin{corollary}\label{cor_all_vertical_rays_are_contracting}
    Let $G$ be a finitely generated group and let $\mathcal{H}=\{H_1,\dots,H_n\}$ be a finite collection of finitely generated infinite subgroups of $G$. If the contracting boundary $\partial_c^{\mathcal{FQ}} G^h$ is a non-empty compact set containing at least three points, then every combinatorial horoball in the cusped space $G^h$ is contracting.
\end{corollary} 

As a consequence of Corollary~\ref{cor_all_vertical_rays_are_contracting} and Theorem~\ref{thm_pal_pandey}, we obtain the following theorem.

\begin{theorem}\label{thm_contracting_compact_implies_RH}
    Let $G$ be a finitely generated group and let $\mathcal{H}=\{H_1,\dots,H_n\}$ be a finite collection of finitely generated infinite subgroups of $G$. If the contracting boundary $\partial_c^{\mathcal{FQ}} G^h$ is a non-empty compact set containing at least three points, then the pair $(G,\mathcal{H})$ is relatively hyperbolic.
\end{theorem} 

For a proper geodesic metric space, if the Morse boundary is compact with respect to the direct limit topology, then the contracting boundary equipped with the fellow-travelling quasigeodesic topology is also compact. Thus, the following theorem is an immediate consequence of Theorem~\ref{thm_contracting_compact_implies_RH}.

\begin{theorem}\label{thm_Morse_compact_implies_RH}
    Let $G$ be a finitely generated group and let $\mathcal{H}=\{H_1,\dots,H_n\}$ be a finite collection of finitely generated infinite subgroups of $G$. If the Morse boundary $\partial_M^{\mathcal{DL}} G^h$ is a non-empty compact set containing at least three points, then the pair $(G,\mathcal{H})$ is relatively hyperbolic.
\end{theorem} 

\begin{remark}
    Theorem~\ref{thm_all_vertical_rays_are_contracting} can also be established in the setting of the Morse boundary equipped with the direct limit topology. This follows from the fact that the weak hull $\mathfrak{h}(\partial_M G^h)$ is hyperbolic in the cusped space $G^h$ whenever $\partial_M^{\mathcal{DL}} G^h$ is compact and contains at least three points (see Proposition~\ref{prop_weakhull_hyperbolic}). 
    In this approach, one applies Theorem~\ref{prop_condition_to_be_rel_hyp} in place of Theorem~\ref{thm_pal_pandey} in Case~2 of the proof of Theorem~\ref{thm_all_vertical_rays_are_contracting}.

    As an immediate consequence of the above observation together with Theorem~\ref{prop_condition_to_be_rel_hyp}, Theorem~\ref{thm_Morse_compact_implies_RH} can be established independently within the framework of the Morse boundary equipped with the direct limit topology, without invoking the contracting boundary or the fellow-travelling quasi-geodesic topology.
\end{remark}

In Theorem~\ref{thm_1}, the implications \((1) \Rightarrow (2)\) and \((1) \Rightarrow (3)\) are immediate. Therefore, Theorem~\ref{thm_contracting_compact_implies_RH} and Theorem~\ref{thm_Morse_compact_implies_RH} together complete the proof of Theorem~\ref{thm_1}.




\bibliographystyle{amsplain}

\end{document}